\documentclass[11pt]{article}

\usepackage{amsmath,amsfonts,amssymb,color,diagrams}

\newcommand{\beq}{\begin{equation}}
\newcommand{\eeq}{\end{equation}}
\newcommand{\bea}{\begin{eqnarray}}
\newcommand{\eea}{\end{eqnarray}}
\newcommand{\beas}{\begin{eqnarray*}}
\newcommand{\eeas}{\end{eqnarray*}}

\definecolor{dg}{rgb}{0, 0.5, 0}

\newtheorem{theorem}{Theorem}[section]

\newtheorem{definition}[theorem]{Definition}
\newtheorem{proposition}[theorem]{Proposition}

\newtheorem{prop}[theorem]{Proposition}

\newtheorem{lemma}[theorem]{Lemma}
\newtheorem{remark}[theorem]{Remark}
\newtheorem{example}[theorem]{Example}
\newtheorem{examples}[theorem]{Examples}
\newtheorem{foo}[theorem]{Remarks}

\newenvironment{proof}{\addvspace{\medskipamount}\par\noindent{\it
Proof}.}
{\unskip\nobreak\hfill$\Box$\par\addvspace{\medskipamount}}

\usepackage{hyperref}
\hypersetup{
    colorlinks=true,
    linkcolor=black,
    filecolor=magenta,      
    urlcolor=cyan,
    citecolor=blue
}

\newcommand{\bH}{\mathbb H}

\newcommand{\Z}{\mathbb Z}

\newcommand{\bS}{\mathbb S}

\newcommand{\R}{\mathbb R}
\newcommand{\C}{\mathbb C}

\newcommand{\B}{\mathbb B}
\newcommand{\A}{\mathbb A}

\title{The horizontal heat kernel on the quaternionic anti de-Sitter spaces and related twistor spaces}
\author{Fabrice Baudoin\footnote{Author supported in part by the NSF Grant DMS 1660031}, Nizar Demni, Jing Wang }

\begin{document}

\maketitle

\begin{abstract}
The geometry of the quaternionic anti-de Sitter fibration  is studied in details. As a consequence, we obtain   formulas for the horizontal Laplacian and subelliptic heat kernel of the fibration. The heat kernel formula is explicit enough to derive small time asymptotics. Related twistor spaces and corresponding heat kernels are also discussed and the connection to the quaternionic magnetic Laplacian is done.
\end{abstract}

\tableofcontents 

\section{Introduction}

The main goal of the paper is to study in details the geometry, the horizontal Laplacians and the horizontal heat kernels of the quaternionic anti-de Sitter fibration
\begin{align}\label{fibration intro}
\mathbf{SU}(2)\to \mathbf{AdS}^{4n+3}(\mathbb{H})\to \bH H^n
\end{align}
and of related twistor spaces. This fibration is the quaternionic counterpart of the complex anti-de Sitter fibration that was studied in details in \cite{BD,Bo,W}.  Anti-de Sitter spaces play a fundamental role in mathematical physics since they appear as  exact solutions of Einstein's field equations for an empty universe with a negative cosmological constant; see for instance \cite{G} and the references therein for a pedagogical account on the importance of those spaces and the celebrated AdS/CFT correspondence.  

Anti-de Sitter spaces also appear as model spaces in Sasakian geometry: the complex anti de-Sitter space is the model space of a negative Sasakian manifold (see \cite{BGKT}) and the quaternionic anti-de Sitter fibration studied in this paper can be thought as the model space of a negative 3-Sasakian manifold (see \cite{BG,Jel}).

\

Let us  describe the main results of the paper. Section \ref{sec-geometry} is devoted to the geometric study of the quaternionic anti-de Sitter fibration. We approach and describe this geometry from two complementary points of view: the point of view of pseudo-Riemannian submersions (see \cite{BA}) and the point of view of quaternionic contact geometry (see \cite{biquard}). Then, an important observation is that $\mathbf{SU}(2)$ and the hyperbolic quaternionic space $\bH H^n$ are both symmetric spaces of rank 1. The fibration \eqref{fibration intro} therefore inherits a  large symmetry group. As a consequence, the heat kernel of the horizontal Laplacian on $\mathbf{AdS}^{4n+3}(\mathbb{H})$ only depends on two variables: a variable $r$ which is the radial coordinate on the base space $\mathbb{H}H^n$ and a variable $\eta$ which is the radial coordinate on the fiber $\mathbf{SU}(2)$.   We prove that in these coordinates,  the  horizontal Laplacian of the fibration writes
\[
\frac{\partial^2}{\partial r^2}+((4n-1)\coth r+3\tanh r)\frac{\partial}{\partial r}+\tanh^2r \left(\frac{\partial^2}{\partial \eta^2}+2\cot \eta\frac{\partial}{\partial \eta}\right).
\]

In section \ref{sec-kernel}, using this expression, we shall derive an integral representation of the corresponding heat kernel $p_t(r, \eta)$ by two different methods.

\begin{itemize}
\item The first method is geometric and uses an analytic continuation in the fiber variables similar to the Wick rotations used in physics;
\item The second method is related to the study of heat kernels associated with the generalized Maass Laplacian (see e.g. \cite{Int-OM}) and uses the ideas from \cite{BD}. 
\end{itemize}

Our main result can be summarized as follows:

\begin{theorem}\label{IntRepintro} 
Let $r \ge 0$ and $ \eta \in [0,\pi)$, then:
 \begin{align}\label{eq-IntRep}
 & p_t(r,\eta)\nonumber\\
 = &\frac2\pi \int_{0}^{\infty} \frac{\sinh u}{\sin \eta}  \left\{\sum_{m \geq 0} e^{-m(m+2)t} \sin[(m+1)\eta] \sinh[(m+1)u] \right\}q_{t, 4n+3}(\cosh r \cosh u) du \nonumber\\
 =&\frac{ e^{t}}{  \sqrt{\pi t}  }  \sum_{k\in\mathbb Z}\int_0^{+\infty} \frac{ \sinh y \sin \left(  \frac{ (\eta+2k\pi) y}{2t}\right) }{\sin \eta} e^{\frac{y^2- (\eta+2k\pi)^2}{4t}} q_{t,4n+3}( \cosh r\cosh y ) dy.
\end{align}
where
\begin{equation*}
q_{t, 4n+3}(\cosh s) := \frac{e^{-(2n+1)^2t}}{(2\pi)^{2n+1}\sqrt{4\pi t}}\left(-\frac{1}{\sinh(s)}\frac{d}{ds}\right)^{2n+1} e^{-s^2/(4t)},
\end{equation*} 
is the heat kernel on the real $4n+3$ dimensional hyperbolic space.
\end{theorem}

The role of the  $4n+3$ dimensional real hyperbolic space in this formula is puzzling. One can explain it heuristically by noting that the Cartan dual of $\mathbf{AdS}^{4n+3}(\mathbb{H})$ obtained by complexification of the fiber $\mathbf{SU}(2)$ can be thought of as the real $4n+3$ dimensional real hyperbolic space (see the long comment after the proof of theorem \ref{IntRep}).  Several consequences of the formula are then discussed. In particular, we obtain small time asymptotics for $p_t(r,\eta)$ and prove that
\[
-\frac{e^{-4nt}}{2\pi\cosh r\sin\eta}\frac{\partial}{\partial\eta}p_t^{\mathbb{C}}(r, \eta)=p_t(r,\eta).
\]
where $p_t^{\mathbb{C}}$ is the horizontal heat kernel of the complex anti de-Sitter fibration as studied in \cite{BD,W}.

In section \ref{sec-twistor}, we study the horizontal Laplacian and corresponding horizontal heat kernel of the twistor space $\mathbb{C}H_1^{2n+1}$ over $\mathbb{H}H^n$. This twistor space appears as a $\mathbb{S}^1$ quotient of the quaternionic anti de-Sitter fibration according to the following commutative diagram:

 \begin{diagram}\label{diag}
  & & \mathbb{S}^1 & & \\
  & \ldTo & \dTo & & \\
 \mathbf{SU}(2) & \rTo &\mathbf{AdS}^{4n+3}(\mathbb{H}) & \rTo & \mathbb{H}H^n \\
 \dTo & &\dTo & \ruTo & \\
 \mathbb{CP}^1 & \rTo & \mathbb{C}H_1^{2n+1} & & 
\end{diagram}

We then prove that  the radial part of the  horizontal Laplacian on  $\mathbb{C}H_1^{2n+1}$ is given by 
\begin{equation}\label{L-CPintro}
\frac{\partial^2}{\partial r^2}+((4n-1)\coth r+3\tanh r)\frac{\partial}{\partial r}+\tanh^2r\left( \frac{\partial^2}{\partial \phi^2}+2\cot 2\phi\frac{\partial}{\partial \phi}\right),
\end{equation}
where $r$ is, once again, the radial coordinate on $\mathbb{H}H^n$ and $\phi$ is the radial coordinate on  the complex projective space $\mathbb{CP}^1$. The corresponding heat kernel writes
\begin{align*}
 & h_t(r,\phi)\\
 = & \int_{0}^{\infty} (\sinh u)^2 \left\{\sum_{m=0}^{+\infty} (2m+1) e^{-4m(m+1)t} P_m^{0,0}(\cos2\phi )P_m^{0,0}(\cosh 2u)\right\}q_{t, 4n+3}(\cosh r \cosh u) du.
\end{align*}
where $P_m^{0,0}$ are  Legendre polynomials.

Finally, in section \ref{sec-further}, we discuss further developments. A connection to the quaternionic magnetic Laplacian is made and an interesting sub-d'Alembertian on the real 4n+3 dimensional hyperbolic space is discussed.

\

\textbf{Acknowledgments:} \textit{The authors would like to thank Brian Hall for general discussions about  complexification of symmetric spaces and the notion of Cartan dual. }

\section{Geometry of the quaternionic anti de-Sitter fibration}\label{sec-geometry}

\subsection{Definition of $\mathbf{AdS}^{4n+3}(\mathbb{H})$}
Let $\mathbb{H}$ be the quaternionic field 
\[
\mathbb{H}=\{q=t+xI+yJ+zK, (t,x,y,z)\in\R^4\},
\]
where  $I,J,K \in \mathbf{SU}(2)$ are the Pauli matrices:
\[
I=\left(
\begin{array}{ll}
i & 0 \\
0&-i 
\end{array}
\right), \quad 
J= \left(
\begin{array}{ll}
0 & 1 \\
-1 &0 
\end{array}
\right), \quad 
K= \left(
\begin{array}{ll}
0 & i \\
i &0 
\end{array}
\right).
\]
Then, the quaternionic norm is given by $|q |^2 =t^2 +x^2+y^2+z^2$ and the set of unit quaternions is identified with $\mathbf{SU}(2)$. 

We consider the quaternionic anti-de Sitter space $\mathbf{AdS}^{4n+3}(\mathbb{H})$ which is defined as a pseudo-hyperbolic space by:
\[
\mathbf{AdS}^{4n+3}(\mathbb{H})=\lbrace q=(q_1,\cdots,q_{n+1})\in \mathbb{H}^{n+1}, \| q \|^2_H =-1\rbrace,
\]
where 
\[
\|q\|_H^2 := \sum_{k=1}^{n}|q_k|^2-|q_{n+1}|^2.
\] 
In real coordinates $q_i = (t_i, x_i, y_i, z_i), 1 \leq i \leq n+1$, this pseudo-norm may be written as
\[
t_1^2+x_1^2+y_1^2+z_1^2+\cdots+t_n^2+x_n^2+y_n^2+z_n^2-t_{n+1}^2-x_{n+1}^2-y_{n+1}^2-z_{n+1}^2=-1, 
\]
and as such, $\mathbf{AdS}^{4n+3}(\mathbb{H})$ is embedded in the flat $4n+4$-dimensional space $\R^{4n,4}$ endowed with the Lorentzian real signature $(4n,4)$ metric 
\begin{equation}\label{eq-lorent-q}
ds^2=dt_1^2+dx_1^2+dy_1^2+dz_1^2+\cdots+dt_n^2+dx_n^2+dy_n^2+dz_n^2-dt_{n+1}^2-dx_{n+1}^2-dy_{n+1}^2-dz_{n+1}^2.
 \end{equation}
As a matter of fact, $\mathbf{AdS}^{4n+3}(\mathbb{H})$ is naturally endowed with a pseudo-Riemannian structure of signature $(4n,3)$. 
 %In fact, any pseudo-hyperbolic space with signature more than $1$ is simply connected (See \cite{BA} Theorem 4.14). %(We can think of $\mathbf{AdS}^{4n+3}(\mathbb{H})$ homeomorphic to $\R^{4n}\times S^3$). 

\subsection{$\mathbf{SU}(2)$ action and fibration}

Using quaternionic multiplication on the left, $\mathbf{SU}(2)$  acts isometrically on $\mathbf{AdS}^{4n+3}(\mathbb{H})$.  In fact, if we consider the rotations $e^{S\theta}$, $S=I, J, K$, then straightforward computations show that the infinitesimal generators are given in real coordinates by
 \begin{equation}\label{eq-1}
\frac{d}{d\theta}f(e^{I\theta}q)\mid_{\theta=0}=\sum_{i=1}^{n+1}\left(-x_i\frac{\partial f}{\partial t_i}+t_i\frac{\partial f}{\partial x_i}-z_i\frac{\partial f}{\partial y_i}+y_i\frac{\partial f}{\partial z_i}\right)=T_1
 \end{equation}
 \begin{equation}\label{eq-2}
 \frac{d}{d\theta}f(e^{J\theta}q)\mid_{\theta=0}= \sum_{i=1}^{n+1}\left(-y_i\frac{\partial f}{\partial t_i}+z_i\frac{\partial f}{\partial x_i}+t_i\frac{\partial f}{\partial y_i}-x_i\frac{\partial f}{\partial z_i}\right)=T_2
 \end{equation}
 \begin{equation}\label{eq-3}
 \frac{d}{d\theta}f(e^{K\theta}q)\mid_{\theta=0}=\sum_{i=1}^{n+1}\left(-z_i\frac{\partial f}{\partial t_i}-y_i\frac{\partial f}{\partial x_i}+x_i\frac{\partial f}{\partial y_i}+t_i\frac{\partial f}{\partial z_i}\right)=T_3
 \end{equation}
 
 The quotient space $\mathbf{AdS}^{4n+3}(\mathbb{H})/ \mathbf{SU}(2)$ can be identified with the quaternionic hyperbolic space $\bH H^n$ endowed with its canonical quaternionic K\"ahler metric. 
 The projection map $\mathbf{AdS}^{4n+3}(\mathbb{H})\to \bH H^n$ is a pseudo-Riemannian submersion with totally geodesic fibers isometric to $\mathbf{SU}(2)$. The fibration
\[
\mathbf{SU}(2)\to \mathbf{AdS}^{4n+3}(\mathbb{H})\to \bH H^n
\]
is referred to as the quaternionic anti de-Sitter fibration and may be thought of as the hyperbolic counterpart of the quaternionic Hopf fibration studied in details in \cite{BW2}. Moreover, it shows that $\mathbf{AdS}^{4n+3}(\mathbb{H})$ is simply connected since both the base space and the fiber are so. We refer to \cite{BA} for a description of this submersion, and more generally for a description of totally geodesic submersions from pseudo-hyperbolic spaces. 

\subsection{Laplacians and horizontal heat kernel}
The horizontal Laplacian $L$ on $\mathbf{AdS}^{4n+3}(\mathbb{H})$ is defined as the horizontal lift of the Laplace-Beltrami operator on $\bH H^n$ by the submersion $\mathbf{AdS}^{4n+3}(\mathbb{H})\to \bH H^n$. Note that since the metric on $\bH H^n$ is Riemannian, the operator $L$ is semi-elliptic (that is, its principal symbol is non negative). More than that, from H\"ormander's theorem it is easily seen to be subelliptic  because the horizontal distribution on $\mathbf{AdS}^{4n+3}(\mathbb{H})$, i.e. the distribution transverse to the fibers of the submersion, is everywhere two-step bracket generating (this comes from the quaternionic contact structure described below). 

Furthermore, according to \cite{BeBo}, one has the decomposition
\[
\square_{\mathbf{AdS}^{4n+3}(\mathbb{H})}=L-\Delta_{\mathbf{SU}(2)} ,
\]
where $\square_{\mathbf{AdS}^{4n+3}(\mathbb{H})}$ is the d'Alembertian on $\mathbf{AdS}^{4n+3}(\mathbb{H})$ i.e. the Laplace-Beltrami operator of the pseudo-Riemannian metric, and $\Delta_{\mathbf{SU}(2)}=T_1^2+T_2^2+T_3^2$ is the Laplace-Beltrami operator of the fibers. Since the fibers are totally geodesic and thus isometric to $\mathbf{SU}(2)$,  $\Delta_{\mathbf{SU}(2)} $ is simply the Laplace-Beltrami operator on $\mathbf{SU}(2)$. 

The Riemannian metric on $\mathbf{AdS}^{4n+3}(\mathbb{H})$, i.e. the one obtained from the pseudo-Riemannian one by simply changing the signature in the direction of the fibers, is complete, thus the horizontal Laplacian is essentially self-adjoint on the space of smooth and compactly supported functions (see Section 5.1 in \cite{BEMS}). As a consequence, $L$ is the generator of a strongly continuous semigroup $e^{tL}$ in $L^2(\mathbf{AdS}^{4n+3}(\mathbb{H}))$ that uniquely solves the heat equation. By subellipticity of $L$, this semigroup admits a heat kernel that we will refer to as the horizontal or subelliptic heat kernel.

\subsection{Quaternionic contact structure}
In addition to the fibration structure on $\mathbf{AdS}^{4n+3}(\mathbb{H})$,  there is a quaternionic contact structure (see \cite{biquard,BG} for a description of this structure). Indeed, if we consider  the quaternionic  form
\begin{align}\label{eq-contact-form}
\alpha = \frac12\left( \sum_{i=1}^n (dq_i\, \overline{q_i}-q_i\,d\overline{q_i})- (dq_{n+1}\,\overline{q_{n+1}}-q_{n+1}\,d\overline{q_{n+1}})\right)=\alpha_1I+\alpha_2J+\alpha_3K,
\end{align}
then the triple $(\alpha_1,\alpha_2,\alpha_3)$ gives the quaternionic contact structure.  If we denote
\[
T=- \sum_{i=1}^n \left(q_i\frac{\partial}{\partial q_i}-\frac{\partial}{\partial \overline{q_i}}\overline{q_i}\right)+\left(q_{n+1}\frac{\partial}{\partial q_{n+1}}-\frac{\partial}{\partial \overline{q_{n+1}}}\overline{q_{n+1}}\right) 
\]
then $T=T_1I+T_2J+T_3K$, $\alpha(T)=3$ and we can easily find that
\[
\alpha_i(T_j)=-\delta_{ij}.
\]
Thus, $T_1$, $T_2$, $T_3$ are the three Reeb vector fields of $\alpha$ and are also Killing vector fields on $\mathbf{AdS}^{4n+3}(\mathbb{H})$. In this way, $\mathbf{AdS}^{4n+3}(\mathbb{H})$ is a negative 3-K contact structure (see \cite{Jel}).

\subsection{Cylindric coordinates}
In this paragraph, we shall write down the horizontal Laplacian in cylindrical coordinates: these consist of coordinates $(w_1,\dots,w_n)$ in the base space $\bH H^n$ and $(\theta_1, \theta_2, \theta_3)$ in the Lie algebra $\mathfrak{su}(2)$ of traceless skew-Hermitian $2\times 2$ matrices. More precisely, we shall consider the map

\begin{eqnarray*}
\bH H^n \times \mathfrak{su}(2) & \rightarrow & \mathbf{AdS}^{4n+3}(\mathbb{H}) \\
(w_1,\dots, w_n, \theta_1, \theta_2, \theta_3) & \mapsto &\left(\frac{e^{ I\theta_1 +J\theta_2 +K\theta_3} w_1}{\sqrt{1-\rho^2}},\cdots,\frac{e^{ I\theta_1 +J\theta_2 +K\theta_3}w_n}{\sqrt{1-\rho^2}},\frac{e^{ I\theta_1 +J\theta_2 +K\theta_3}}{\sqrt{1-\rho^2}} \right)
\end{eqnarray*}
%and restrict to 
%\[
%\kappa=1,\quad \frac{\partial}{\partial\kappa}=0,
%\]
where 
\begin{equation*}
\rho = \sqrt{\sum_{j=1}^{n}|w_j|^2}
\end{equation*}
and $w_i = q_i/q_{n+1}$, $i=1,\dots, n,$ are inhomogeneous coordinates in $\bH H^n$. Accordingly, 
\[
dq_i=q_{n+1}dw_i+dq_{n+1}w_i, \quad d\overline{q_i}=\overline{dw_i}\overline{q_{n+1}}+\overline{w_i}\overline{dq_{n+1}}, 
\]
and we can then rewrite the contact form \eqref{eq-contact-form} as follows:
\begin{align*}
\alpha & = \frac12\left( \sum_{i=1}^n ((q_{n+1}dw_i+dq_{n+1}w_i)\, \overline{w_i}\,\overline{q_{n+1}}-q_{n+1}w_i\,(d\overline{w_i}\,\overline{q_{n+1}}+\overline{w_i}d\overline{q_{n+1}}))- (dq_{n+1}\,\overline{q_{n+1}}-q_{n+1}\,d\overline{q_{n+1}})\right)\\
&=\frac12\left( \sum_{i=1}^n q_{n+1}dw_i\, \overline{w_i}\,\overline{q_{n+1}}-q_{n+1}w_i\,d\overline{w_i}\,\overline{q_{n+1}}\right)
-\frac12\left(1-\rho^2 \right)\left(dq_{n+1}\overline{q_{n+1}}-q_{n+1}d\overline{q_{n+1}}   \right).
\end{align*}
Set ${\frak{q}} := { I\theta_1 +J\theta_2 +K\theta_3}$, then 
\begin{equation*}
q_{n+1}=\frac{e^{\frak{q}}}{\sqrt{1-\rho^2}},  \quad \overline{e^{\frak{q}}}=e^{-{\frak{q}}}, \quad \overline{{\frak{q}}}=-{\frak{q}}={\frak{q}}^{-1}\eta^2, 
\end{equation*}
where $\eta^2=\theta^2_1+\theta^2_2+\theta^2_3$ is the squared Riemannian distance from the identity in $SU(2)$. Also, the relation $0=d({\frak{q}}\,{\frak{q}}^{-1})={\frak{q}}d{\frak{q}}^{-1}+d{\frak{q}}\, {\frak{q}}^{-1}$ yields
\[
d{\frak{q}}^{-1}=-{\frak{q}}^{-1}d{\frak{q}}\cdot {\frak{q}}^{-1}, 
\]
%Moreover one can easily compute that $dq\cdot q=\overline{qdq}=d\overline{q}\cdot\overline{q}$. And $-q^2=|q|^2=\eta^2$ implies that $dq\cdot q^2= q^2dq$. Now,
%since $A_k:=dq^k=\sum_{j=0}^{k-1}q^jdq\, q^{k-1-j}$, we have for $k=0,1,2, \cdots$ 
%\begin{align*}
%&A_{2k+1}=(k+1)q^{2k}dq+kq^{2k-1}dq\cdot q=q^{2k}dq-kq^{2k-1}d\eta^2 \\
%&A_{2k}=kq^{2k-1}dq+kq^{2k-2}dq\cdot q=-kq^{2k-2}d\eta^2
%\end{align*}
%Here we used the fact that $dq\cdot q+qdq=-d\eta^2=-2(\theta_1d\theta_1+\theta_2d\theta_2+\theta_3d\theta_3)$. 
Now, we need to compute $de^{\frak{q}}\cdot e^{-{\frak{q}}}$. To this end, recall that 
\begin{equation*}
e^{\frak{q}}=\cos\eta+\frac{\sin\eta}{\eta}(\theta_1I+\theta_2J+\theta_3K).
\end{equation*}
Consequently, 
\begin{align*}
de^{\frak{q}}\cdot e^{-{\frak{q}}}=\left(d(\cos\eta)+d\left( \frac{\sin\eta}{\eta}{\frak{q}}\right)\right)\left( \cos\eta-\frac{\sin\eta}{\eta}{\frak{q}}\right)
%de^q&=\sum_{k=0}^\infty \frac{1}{k!}A_k=\sum_{k=1}^\infty \frac{1}{(2k)!}\left(-kq^{2k-2} \right)d\eta^2+\sum_{k=0}^\infty \frac{1}{(2k+1)!}\left(q^{2k}dq-kq^{2k-1}d\eta^2 \right)\\
%&=\frac{e^q-e^{-q}}{4}q^{-1}d\eta^2+\frac{e^q-e^{-q}}{2}q^{-1}dq-\frac{e^q+e^{-q}}{4}q^{-1}d\eta^2+\frac{e^q-e^{-q}}{4}q^{-2}d\eta^2\\
%&= \frac{e^q-e^{-q}}{4}q^{-1}(dq-q^{-1}dq \cdot q)+\frac{e^{-q}}{2}(dq+q^{-1}dq \cdot q)
\end{align*}
and 
\begin{align*}
de^{\frak{q}}\cdot e^{-{\frak{q}}}-e^{\frak{q}}\cdot de^{-{\frak{q}}}&=\cos\eta\, d\left( \frac{\sin\eta}{\eta}\right){\frak{q}}-\frac{\sin\eta}{\eta}{\frak{q}}\, d\cos\eta\\
&=\cos^2\eta\, d\left( \frac{\tan\eta}{\eta}{\frak{q}} \right)
\end{align*}
%Hence $dq_{n+1}=\frac{de^q}{\sqrt{1-\rho^2}}+e^{q}d\frac{1}{\sqrt{1-\rho^2}}$ and  $dq_{n+1}\overline{q_{n+1}}=\frac{de^qe^{-q}}{{1-\rho^2}}+\frac{1}{\sqrt{1-\rho^2}}d\frac{1}{\sqrt{1-\rho^2}}$. Recall that $e^q=\cos\eta+\frac{\sin\eta}{\eta}(\theta_1I+\theta_2J+\theta_3K)$, hence
%\[
%dq_{n+1}\overline{q_{n+1}}=\frac{1}{{1-\rho^2}}\left( \frac{1-e^{-2q}}{4}q^{-1}(dq-q^{-1}dq \cdot q)+\frac{e^{-2q}}{2}(dq+q^{-1}dq \cdot q)\right)+\frac{1}{2}d\frac{1}{{1-\rho^2}}
%\]
%and
%\begin{align*}
%dq_{n+1}\overline{q_{n+1}}-q_{n+1}d\overline{q_{n+1}}&=\frac{1}{{1-\rho^2}}\left( \frac{e^{2q}-e^{-2q}}{4}q^{-1}(dq-q^{-1}dq \cdot q)+\frac{e^{2q}+e^{-2q}}{2}(dq+q^{-1}dq \cdot q)\right)\\
%&=\frac{1}{{1-\rho^2}}\left(\frac{\sin2\eta}{2\eta}(2dq+q^{-1}d\eta^2)-\cos2\eta\cdot q^{-1}d\eta^2 \right)\\
%&=\frac{1}{{1-\rho^2}}\,d\left( \frac{\sin2\eta}{\eta} q\right)
%\end{align*}
As a result, 
\[
\alpha=\frac12e^{\frak{q}}\left( \frac{dw_i\, \overline{w_i}-w_i\,d\overline{w_i}}{1-\rho^2}-\cos^2\eta\, d\left( \frac{\tan\eta}{\eta}{\frak{q}} \right)\right)e^{-{\frak{q}}}.
\]
Equivalently, we can consider the following one form
\[
\Lambda:=e^{-{\frak{q}}}\,\alpha\, e^{\frak{q}}=\frac12\left( \frac{dw_i\, \overline{w_i}-w_i\,d\overline{w_i}}{1-\rho^2}-\cos^2\eta\, d\left( \frac{\tan\eta}{\eta}{\frak{q}} \right)\right)
\]
whose horizontal part 
\begin{equation}\label{QKah}
\zeta := \frac12\left( \frac{dw_i\, \overline{w_i}-w_i\,d\overline{w_i}}{1-\rho^2}\right)
\end{equation}
 is the quaternionic K\"ahler form on $\bH H^n$, which induces the following sub-Riemannian metric
\begin{equation}\label{eq-metric}
h_{i\bar{k}}=\frac12\left(\frac{\delta_{i{k}}}{1-\rho^2}+\frac{\overline{w_i}w_k}{(1-\rho^2)^2}\right).
\end{equation}
We are now ready to derive the sub-Laplacian on $\mathbf{AdS}^{4n+3}(\mathbb{H})$. Due to radial symmetry, we are interested in the radial part of the sub-Laplacian, in the sense as follows.
\begin{definition} Let $\psi$ be the map from $\mathbf{AdS}^{4n+3}(\mathbb{H}) $ to $ [0,1) \times [0,\pi) $ such that 
\[
\psi  \left(\frac{e^{ I\theta_1 +J\theta_2 +K\theta_3} w_1}{\sqrt{1-\rho^2}},\cdots,\frac{e^{ I\theta_1 +J\theta_2 +K\theta_3}w_n}{\sqrt{1-\rho^2}},\frac{e^{ I\theta_1 +J\theta_2 +K\theta_3}}{\sqrt{1-\rho^2}} \right)=\left(\rho, \eta \right),
\]
 We denote by $\mathcal{D}$ the space of smooth and compactly supported functions on $  [0,1) \times [0,\pi) $. Then the cylindrical part of $L$ is defined by $\tilde{L}:\mathcal{D}\to C^\infty( \mathbf{AdS}^{4n+3}(\mathbb{H}))$ such that for every  $f \in \mathcal{D}$, we have
\[
L(f \circ \psi)=(\tilde{L} f) \circ \psi.
\]
\end{definition}

\begin{proposition}
The radial part of the sub-Laplacian on $\mathbf{AdS}^{4n+3}(\mathbb{H})$ is given in the coordinates $(r,\eta)$ by
\begin{equation}\label{radial-L}
\tilde{L}=\frac{\partial^2}{\partial r^2}+((4n-1)\coth r+3\tanh r)\frac{\partial}{\partial r}+\tanh^2r \left(\frac{\partial^2}{\partial \eta^2}+2\cot \eta\frac{\partial}{\partial \eta}\right).
\end{equation}
where $r=\tanh^{-1}\rho$ is the Riemannian distance on $\bH H^n$ from the origin.
\end{proposition}
\begin{proof}
Lifting the vector fields $\partial/\partial w_i, 1 \leq i \leq n,$ by means of $\Lambda$, we first obtain a basis $(V_i)_{1 \leq i \leq n}$ of the horizontal bundle $H(\mathbf{AdS}^{4n+3}(\mathbb{H}))$: 
\begin{equation*}
V_i := \frac{\partial}{\partial w_i} + \frac{\overline{w_i}}{2(1-\rho^2)\cos^2\eta}\frac{\partial}{\partial \phi},
\end{equation*}
where 
\begin{equation*}
\phi=\frac{\tan \eta}{\eta}{\frak{q}}:= \phi_1I+\phi_2J+\phi_3K
\end{equation*}
and 
\begin{equation*}
\frac{\partial}{\partial \phi} = -\left( \frac{\partial}{\partial \phi_1} I + \frac{\partial}{\partial \phi_2} J + \frac{\partial}{\partial \phi_3}K\right).
\end{equation*}
%It then provides us the horizontal sub-Laplacian $L$ on $\mathbf{AdS}^{4n+3}(\mathbb{H})$. 
Moreover, the inverse of the metric displayed in \eqref{eq-metric} is given by: 
\begin{equation*}
h^{\bar{i} k} = 2(1-\rho^2) (\delta_{ik} - \overline{w_i} w_k)
\end{equation*}
and the sublaplacian admits the following expression: 
\begin{align*}
L = 2\sum_{i,k=1}^n \Re(\overline{V_i}h^{\bar{i}k}V_k).   
\end{align*}
Straightforward computations then yield
\begin{align}\label{SubLapl}
L & = 4(1-\rho^2)\Re \left(\sum_{i=1}^n \frac{\partial^2}{\partial\overline{w_i}\partial w_i} - \overline{\mathcal{R}}\mathcal{R} - \frac{\rho^2}{4(1-\rho^2)\cos^4\eta} \left(\frac{\partial}{\partial \phi}\right)^2+ 
\frac{1}{2\cos^2\eta}\left(\overline{\mathcal{R}}\frac{\partial}{\partial \phi} - \frac{\partial}{\partial \phi}\mathcal{R}\right) \right) \nonumber 
\\& = \Delta_{\bH H^n}+\frac{\rho^2}{\cos^4\eta}\sum_{i=1}^3\frac{\partial^2}{\partial \phi_i^2} + \frac{2(1-\rho^2)}{\cos^2\eta} \left(\overline{\mathcal{R}}\frac{\partial}{\partial \phi} - \frac{\partial}{\partial \phi}\mathcal{R}\right), 
\end{align}
where 
\begin{equation*}
\Delta_{\bH H^n} =  4(1-\rho^2)\Re \left(\sum_{i=1}^n \frac{\partial^2}{\partial\overline{w_i}\partial w_i} - \overline{\mathcal{R}}\mathcal{R}\right)
\end{equation*}
is the Laplace-Beltrami operator of ${\bH H^n}$ and 
\begin{equation*}
\mathcal{R} = \sum_{i=1}^n w_i \frac{\partial }{\partial w_i}
\end{equation*}
is the quaternionic Euler operator. Denote $\mu^2 =\phi_1^2+\phi_2^2+\phi_3^2$, then $\mu=\tan \eta$ so that
\begin{equation*}
\frac{\partial}{\partial \eta}=\frac{1}{\cos^2\eta}\frac{\partial}{\partial \mu}.
\end{equation*}
As a result, for any smooth function $f$ of $\eta$,
\begin{align*}
&\frac{1}{\cos^4\eta}\sum_{i=1}^3\frac{\partial^2}{\partial \phi_i^2}f(\eta)=\frac{1}{\cos^4\eta}\left(\frac{\partial^2}{\partial \mu^2}+\frac{2}{\mu}\frac{\partial^2}{\partial \mu^2}\right)f(\eta)\\
&\quad\quad=\left(\frac{\partial^2}{\partial \eta^2}-\cos^2\eta\frac{\partial}{\partial \eta}\left(\frac{1}{\cos^2\eta} \right)\frac{\partial}{\partial \eta}+\frac{2\cot\eta}{\cos^2\eta}\frac{\partial}{\partial \eta} \right)f(\eta)\\
&\quad\quad=\left(\frac{\partial^2}{\partial \eta^2}+2\cot\eta\frac{\partial}{\partial \eta}  \right)f(\eta)
\end{align*}
Since the radial part of $\Delta_{\bH H^n}$ is given by
\[
\frac{\partial^2}{\partial r^2}+((4n-1)\coth r+3\tanh r)\frac{\partial}{\partial r^2}.%+\tanh^2 r\left(\frac{\partial^2}{\partial \eta^2}+2\cot\eta\frac{\partial}{\partial \eta}  \right).
\]
and $\rho = \tanh(r)$, the proposition is proved.
\end{proof}

\section{The subelliptic heat kernel on the quaternionic anti de-Sitter space}\label{sec-kernel}

Recall from the previous section, the radial part of the horizontal Laplacian of the quaternionic anti de-Sitter fibration
\[
\mathbf{SU}(2)\to \mathbf{AdS}^{4n+3}(\mathbb{H})\to \bH H^n
\]
is given by
\begin{equation*}
 \frac{\partial^2}{\partial r^2}+((4n-1)\coth r+3\tanh r)\frac{\partial}{\partial r^2}+\tanh^2 r\left(\frac{\partial^2}{\partial \eta^2}+2\cot\eta\frac{\partial}{\partial \eta}  \right),
\end{equation*}
where $r \ge 0$ is the radial coordinate on the base space  $\bH H^n$ and $\eta \in [0,\pi]$ is the radial coordinate on $\mathbf{SU}(2)$. With a slight abuse of notation with respect to the previous section, this operator shall be denoted by $L$. Let $p_t(r,\eta)$ be its subelliptic heat kernel  (issued from $(0,0)$) with respect to the Riemannian volume measure 
\begin{equation*}
\nu (dr, d\eta) = \frac{8\pi^{2n+1}}{\Gamma(2n)}(\sinh r)^{4n-1}(\cosh r)^3(\sin\eta)^2drd\eta.
\end{equation*}

In this section, we shall derive an integral representation of the corresponding heat kernel $p_t(r, \eta)$ by two different methods.

\begin{itemize}
\item The first method is geometric and uses an analytic continuation in the fiber variables similar to the Wick rotations;
\item The second method is related to the heat kernel associated with the generalized Maass Laplacian (see e.g. \cite{Int-OM}). 
\end{itemize}

\subsection{First method: Complexification of $\mathbf{SU}(2)$ and Wick rotations}

%We now use the above considerations at the level of the operators. Note that from radial symmetries of the heat kernels,  we may  restrict ourselves to the radial parts of the operators.% however from the previous remarks the considerations below extend to the operators (and underlying metrics) themselves.

\begin{theorem}\label{IntRep} 
Let $r \ge 0$ and $ \eta \in [0,\pi)$, then:
 \begin{align}\label{eq-IntRep}
 & p_t(r,\eta)\nonumber\\
 = &\frac2\pi \int_{0}^{\infty} \frac{\sinh u}{\sin \eta}  \left\{\sum_{m \geq 0} e^{-m(m+2)t} \sin[(m+1)\eta] \sinh[(m+1)u] \right\}q_{t, 4n+3}(\cosh r \cosh u) du \nonumber\\
 =&\frac{ e^{t}}{  \sqrt{\pi t}  }  \sum_{k\in\mathbb Z}\int_0^{+\infty} \frac{ \sinh y \sin \left(  \frac{ (\eta+2k\pi) y}{2t}\right) }{\sin \eta} e^{\frac{y^2- (\eta+2k\pi)^2}{4t}} q_{t,4n+3}( \cosh r\cosh y ) dy.
\end{align}
where
\begin{equation*}
q_{t, 4n+3}(\cosh s) := \frac{e^{-(2n+1)^2t}}{(2\pi)^{2n+1}\sqrt{4\pi t}}\left(-\frac{1}{\sinh(s)}\frac{d}{ds}\right)^{2n+1} e^{-s^2/(4t)},
\end{equation*} 
is the heat kernel on the real $4n+3$ dimensional real hyperbolic space.
\end{theorem}
%{\color{red}{In the above formulas for $p_t(r,\eta)$, one may need to adjust the normalization constant with respect to the measure $d\mu$. Fabrice}}
%
%{\color{blue}The normalization constant here should be $\frac{2}{\pi}$ in order to have 
%\[
%\int p_t(r,\eta)d\mu(r,\eta)=1
%\]
%Convention of normalization: \\
%\begin{itemize}
%\item[(1)] $q_{t,4n+3}(\cosh s)$ is such that  
%$
%\int q_{t,4n+3}(\cosh s) d\mu_{4n+3}(s)=1
%$
%where 
%\[
%d\mu_{4n+3}(s)=\frac{2^{4n+2}\pi^{2n+1}\Gamma(2n+1)}{\Gamma(4n+2)}\sinh^{4n+2}(s)ds,  \quad s\in[0,\infty)
%\]
%\item[(2)] The above volume measure is consistent with the cylindrical version
%\[
%d\mu_{4n+3}(r,\eta)=\frac{8\pi^{2n+1}}{\Gamma(2n)}(\sinh r)^{4n-1}(\cosh r)^3(\sin\eta)^2drd\eta,\quad r\in[0,\infty),\ \eta\in [0,\pi),
%\]
%and the compact versions
%\begin{align*}
%d\mu^{\bS}_{4n+3}(s)&=\frac{2^{4n+2}\pi^{2n+1}\Gamma(2n+1)}{\Gamma(4n+2)}\sin^{4n+2}(s)ds,  \quad s\in[0,\pi/2)\\
%&=\frac{8\pi^{2n+1}}{\Gamma(2n)}(\sin r)^{4n-1}(\cos r)^3(\sin\eta)^2drd\eta,\quad r\in[0,\pi/2),\ \eta\in [0,\pi),
%\end{align*}
%\item[(3)] We can possibly choose the normalization of $s_t(\eta, u)$ following the above convention as the compact case of $q^{\bS}_{t,3}$ when $n=0$. Then
%\[
% s_t (\eta,u) =\frac{1}{2\pi^2 \sin \eta \sin u} \sum_{m \geq 0} e^{-m(m+2)t} \sin[(m+1)\eta] \sin[(m+1)u] 
%\]
%with volume measure $d\mu_3(u)=4\pi \sin^2udu$.
%This implies that using the expressions in \eqref{eq-s-t} and the convolution in \eqref{formula 1}, we need to add a constant $\frac{2}{\pi}$.
%\end{itemize}
%}

\begin{proof}
%{\color{red}{We first omit the normalization constants }}
The following computations are based on geometric ideas that we will describe after the proof. For conciseness, we omit some of the technical details since a second proof of the result will be given in the next section. We first decompose
\begin{equation*}
L = \frac{\partial^2}{\partial r^2}+((4n-1)\coth r+3\tanh r)\frac{\partial}{\partial r^2}+\tanh^2 r\left(\frac{\partial^2}{\partial \eta^2}+2\cot\eta\frac{\partial}{\partial \eta}  \right).
\end{equation*}
 as follows

\begin{equation*}
L = \Delta_{\bH H^n}+\tanh^2 r {\Delta}_{\mathbf{SU}(2)},
\end{equation*}

where 
\[
\Delta_{\bH H^n}=\frac{\partial^2}{\partial r^2}+((4n-1)\coth r+3\tanh r)\frac{\partial}{\partial r^2}
\]
 is the radial part of the Laplacian on the quaternionic hyperbolic space $\bH H^n$ and 
 \[
 \Delta_{\mathbf{SU}(2)}=\frac{\partial^2}{\partial \eta^2}+2\cot\eta\frac{\partial}{\partial \eta} 
 \]
is the radial part of the Laplacian on $\mathbf{SU}(2)$. Note that we can also write
\[
L={\square}_{\mathbf{AdS}^{4n+3}(\mathbb{H})}+{\Delta}_{\mathbf{SU}(2)},
\]
where
\[
{\square}_{\mathbf{AdS}^{4n+3}(\mathbb{H})}=\Delta_{\bH H^n}-\frac{1}{\cosh^2 r}{\Delta}_{\mathbf{SU}(2)}
\]
is the radial part of the d'Alembertian. Note $\Delta_{\bH H^n}$ and ${\Delta}_{\mathbf{SU}(2)}$ commute. Therefore
\begin{align*}
e^{tL}&=e^{t ({\square}_{\mathbf{AdS}^{4n+3}(\mathbb{H})}+{\Delta}_{\mathbf{SU}(2)})}\\
 &=e^{t {\Delta}_{\mathbf{SU}(2)}} e^{ t {\square}_{\mathbf{AdS}^{4n+3}(\mathbb{H})} }
\end{align*}
We deduce that the heat kernel of $L$ can be written as
\begin{align}\label{formula 1}
p_t(r,\eta)=\int_0^\pi s_t (\eta ,u) p_t^{{\square}_{\mathbf{AdS}^{4n+3}(\mathbb{H})}} (r,u)  (\sin u)^2 du,
\end{align}
where $s_t$ is the heat kernel of 
\[
 \Delta_{\mathbf{SU}(2)}=\frac{\partial^2}{\partial \eta^2}+2\cot\eta\frac{\partial}{\partial \eta} 
 \]
with respect to the measure $ \sin^2\eta d\eta$, $\eta\in[0,\pi)$,
 and $p_t^{{\square}_{\mathbf{AdS}^{4n+3}(\mathbb{H})}} (r,u) $ the heat kernel at $(0,0)$ of ${\square}_{\mathbf{AdS}^{4n+3}(\mathbb{H})}$ with respect to the measure 
 \[
 d\mu_{4n+3}(r,u)=\frac{8\pi^{2n+1}}{\Gamma(2n)}(\sinh r)^{4n-1}(\cosh r)^3(\sin u)^2drdu,\quad  r\in[0,\infty),\ u\in [0,\pi),
 \]
 The idea to compute \eqref{formula 1} is to perform an analytic extension in the fiber variable  $u$. More precisely,   let us consider the analytic change of variables $\tau : (r,\eta) \to (r,i\eta)$  that will be applied on functions of the type $f(r,\eta)=g(r) e^{-i\lambda u}$, with $g$ smooth and compactly supported on $[0,+\infty)$ and $\lambda >0$. One sees then  that
\begin{align}\label{polinh}
{\square}_{\mathbf{AdS}^{4n+3}(\mathbb{H})} (f\circ \tau)=(\Delta_{H^{4n+3}} f ) \circ \tau
\end{align}
where
\[
\Delta_{H^{4n+3}}=\Delta_{\bH H^n}+\frac{1}{\cosh^2 r}{\Delta}_{P}
\]
and
\[
\Delta_{P}=\frac{\partial^2}{\partial \eta^2}+2\coth \eta \frac{\partial}{\partial \eta}.
\]
Introducing a new variable $\delta$ such that $\cosh \delta= \cosh r \cosh \eta$, one sees after straightforward computations that
\[
\Delta_{H^{4n+3}}=\frac{\partial^2}{\partial \delta^2}+(4n+2)\coth\delta\frac{\partial}{\partial \delta}.
\]

Thus $\Delta_{H^{4n+3}}$ is the radial part of the Laplacian on the real hyperbolic space of dimension $4n+3$. One deduces
\begin{align*}
e^{tL} (f \circ \tau)&=e^{t {\Delta}_{\mathbf{SU}(2)}} e^{ t {\square}_{\mathbf{AdS}^{4n+3}(\mathbb{H})} } (f \circ \tau) \\
 &= e^{t {\Delta}_{\mathbf{SU}(2)}} ( (e^{t\Delta_{H^{4n+3}} }f) \circ \tau) \\
 &=(e^{-t \Delta_{P} }e^{t\Delta_{H^{4n+3}}} f ) \circ \tau.
\end{align*}

  Now, since for every  $f(r,\eta)=g(r) e^{-i\lambda u}$ %in the domain of ${\square}_{\mathbf{AdS}^{4n+3}(\mathbb{H})}$,
 \[
( e^{ t {\square}_{\mathbf{AdS}^{4n+3}(\mathbb{H})} } f ) (0,0)=  (e^{t\Delta_{H^{4n+3}} }(f\circ  \tau^{-1})) (0,0),
 \]
 one deduces  that for a function $g$ depending only on $u$, namely $g(u)=e^{-i\lambda u}$
  \[
 \int_0^\pi g(u) p_t^{{\square}_{\mathbf{AdS}^{4n+3}(\mathbb{H})}} (r,u)  (\sin u)^2 du=\int_0^{+\infty} q_{t, 4n+3}(\cosh r \cosh u)g(-iu) (\sinh u)^2 du.
 \]

 Therefore, coming back to \eqref{formula 1}, one infers that using the analytic extension  of $s_t$ one must have
 \[
 \int_0^\pi s_t (\eta ,u) p_t^{{\square}_{\mathbf{AdS}^{4n+3}(\mathbb{H})}} (r,u)  (\sin u)^2 du=\int_0^{+\infty} q_{t, 4n+3}(\cosh r \cosh u) s_t (\eta,-iu) (\sinh u)^2 du.
 \]
 One concludes with the well known formulas (see Section 8.6 in  \cite{F})
 \begin{align}\label{eq-s-t}
 s_t (\eta,u) & =\frac{2}{\pi \sin \eta \sin u} \sum_{m \geq 0} e^{-m(m+2)t} \sin[(m+1)\eta] \sin[(m+1)u] \\
  &=\frac{e^{t}}{  \sqrt{\pi t}  \sin \eta \sin u } \sum_{k\in\Z} \sinh \left(  \frac{ (\eta+2k\pi) u}{2t}\right) e^{\frac{-u^2-(\eta+2k\pi)^2}{4t}}. \nonumber
 \end{align}
\end{proof}

%{\color{red} {

%Let $\tau $ be the change of variable $r \to ir$. We want to prove in an elegant way that
%\[
%(e^{t \Delta_{SU(2)} }f) \circ \tau = e^{-t \Delta_{H^3 } }(f \circ \tau)
%\]

% }}

We now briefly explain the geometric meaning of formula \eqref{polinh}. The idea behind the change of variable $\tau$  is similar to the idea of ``Wick rotation" in physics. More precisely,  consider the complexification of $\mathbf{SU}(2)$, which is given by $SL(2,\C)$, whose Lie algebra $\mathfrak{sl}(2,\C)$ consists of traceless complex matrices. A basis of $\mathfrak{sl}(2,\C)$ is given by
\begin{align*}
&I=\begin{pmatrix} 0 & i\\ i & 0
\end{pmatrix},
\quad 
J=\begin{pmatrix} 0 & 1\\ -1 & 0
\end{pmatrix},
\quad
K=\begin{pmatrix} -i & 0\\ 0 & i
\end{pmatrix}\\
&\mathfrak{I}=\begin{pmatrix} 0 & 1\\ 1 & 0
\end{pmatrix},
\quad 
\mathfrak{J}=\begin{pmatrix} 0 & -i\\ i & 0
\end{pmatrix},
\quad
\mathfrak{K}=\begin{pmatrix} -1 & 0\\ 0 & 1
\end{pmatrix}
\end{align*}
The Lie algebra of $\mathbf{SU}(2)$ consists of traceless complex anti-Hermitian matrices. A basis of $\mathfrak{su}(2)$ is given by $I, J, K$. Clearly we have
\begin{align*}
&[I, J]=2K,\quad [J,K]=2I,\quad [K, I]=2J.
\end{align*}
and 
\[
\mathfrak{sl}(2,\C)=\mathfrak{su}(2)\oplus i\cdot\mathfrak{su}(2).
\] 
Note that $SL(2,\C)$ is isomorphic (as a real analytic variety) to $\mathbf{SU}(2)\times \R^3$ (See Theorem 5, page 21 in \cite{SE}).
The Lie algebra of $SL(2,\R)$ consists of traceless real matrices and a basis is given by $\mathfrak{I}, J, \mathfrak{K}$ and 
\[
[\mathfrak{I}, J]=2\mathfrak{K},\quad [J, \mathfrak{K}]=2\mathfrak{I},\quad [\mathfrak{K}, \mathfrak{I}]=-2J.
\]

%Note $\mathfrak{sl}(2,\C)$ is a complexification of $\mathfrak{su}(2)$
%hence the analytic continuation of $\mathbf{AdS}^{4n+3}(\mathbb{H})$ is a $SL(2, \C)$ bundle on $\bH H^{n}$. 
Denote $\mathfrak{p}:=i\cdot \mathfrak{su}(2)$ and $P=SL(2, \C)/ \mathbf{SU}(2)$. %As a natural generalization of the complex anti de-Sitter case, we consider the $P$-fiber bundle on $\bH H^{n}$ and denoted it by ${H}^{4n+3}$. 
%{\color{blue}
%B. Hall's comments:
%
%Complexify fiberwise: embed original manifold in a $SL(2,\C)$ bundle. Then we have a $P$-bundle inside the $SL(2,\C)$-bundle,
%\[
%P\cong SL(2,\C)/\mathbb{SU}(2).
%\]
%Then analytically continue metric tensor fiberwise, which might give us the real hyperbolic space $H^{4n+3}$.\\
%}
To geometrically describe $P$, we recall the complexification of  $\mathbf{SU}(2)\cong S^3$. Let $Q^3:=\{z=(z_1,z_2, z_3, z_4)\in\C^4, z_1^2+z_2^2+z_3^2+z_4^2=1\}$ be the complexification of $S^3$. It is well known that $Q^3$ can be considered as a tangent bundle $T(S^3):=\{(x,y)|x\in \R^4, y\in\R^4, |x|=1, x\cdot y=0\}$ over $S^3$, using the map $T(S^3)\to Q^3$,
\[
(x,y)\to \cosh |y|x+i\frac{\sinh |y|}{|y|}y.
\]
Take a metric tensor on $S^3$, we analytically continue it to $Q^3$ and then restrict to the fibers. It then gives a hyperbolic space $H^3:=\{y_1^2+y_2^2+y_3^2+y_4^2=-1\}$ with metric $ds^2=-dy_1^2-dy_2^2-dy_3^2-dy_4^2$. 

As a consequence, the complexification of $\mathbf{SU}(2)$ can be considered as fiber bundle over $S^3$ with fibers being real hyperbolic space $H^3$ and the symmetric space  $P=SL(2, \C)/ \mathbf{SU}(2)$ is therefore isometric to $H^3$.

For any analytic function $f(\theta_1, \theta_2, \theta_3)\in L^2(\mathbf{SU}(2))$, we consider its extension to holomorphic functions $f_{\C}(\theta_1+i\mu_1, \theta_2+i\mu_2, \theta_3+i\mu_3)$ on $\mathbf{SU}(2)_\C\cong SL(2, \C)$. %This is given by the Segal-Bargmann transform, which is an isometry from $L^2(\mathbf{SU}(2))$ onto the space of holomorphic functions on $SL(2, \C)$ which are square integrable with respect to a natural measure. (\cite{Stenzel}, see description in \cite{Hall}, page 21).  
It is then easy to see that
\[
(I^2+J^2+K^2) f_{\C}=-(\mathfrak{I}^2+\mathfrak{J}^2+\mathfrak{K}^2) f_{\C}
\]  
Hence the push forward of $\Delta_{\mathbf{\mathbf{SU}(2)}}$ through Wick's rotations gives $-\Delta_P$ where $\Delta_P:=\mathfrak{I}^2+\mathfrak{J}^2+\mathfrak{K}^2$ is the Laplacian on the symmetric space $P=H^3=SL(2, \C)/ \mathbf{SU}(2)$.  Formula  \eqref{polinh} indicates then that the $4n+3$ real dimensional real hyperbolic space is dual to $\mathbf{AdS}^{4n+3}(\mathbb{H})$ by  complexification of the fibers.

We point out that however, there is no (pseudo-) Riemannian submersion from the real hyperbolic space $H^{4n+3}$ to the symmetric space $\mathbb{H}H^n$. To obtain a fibration structure, one can consider the rotation
\[
(I, J, K)\to (\mathfrak{I},J, \mathfrak{K})
\]
The new basis $(J, \mathfrak{K}, \mathfrak{I})$ satisfies 
\[
J^2=-1, \quad \mathfrak{I}^2=\mathfrak{K}^2=1,\quad J\mathfrak{K}=\mathfrak{I},\quad \mathfrak{K}\mathfrak{I}=-J,\quad \mathfrak{I}J=\mathfrak{K}
\]
and form a  split-quaternion algebra $\B$. The resulting space is then a  para-quaternionic hyperbolic space. The quaternionic anti-de Sitter fibration then becomes a pseudo-Riemannian submersion (see \cite{BA})
\[
H^3_1\to H^{4n+3}_{2n+1}\to \B H^n
\] 
with totally geodesic fibers $H^3_1:=\{y_1^2+y_2^2+y_3^2-y_4^2=-1\}$, where 
\[
 H^{4n+3}_{2n+1}=Sp(1,n,\B)/Sp(n, \B),\quad \B H^n=Sp(1,n,\B)/Sp(1,\B)Sp(n,\B).
 \]
Similar rotation can be considered in the complex case as well. In \cite{W}, the author attempted to obtain a real hyperbolic space  by Wick rotating the $\bS^1$ fiber of the complex anti de-Sitter fibration. The correct rotation should be 
$i\to j$,
where $j$ is the Lorentz number that satisfies $j^2=1$ and generates a para-complex algebra $\A$. The complex anti-de Sitter fibration then becomes a pseudo-Riemannian submersion 
\[
H^1\to H^{2n+1}_n\to \A H^n
\]
with totally geodesic fiber $H^1=\{t=x+jy\in\A, t\overline{t}=1, x>0\}$, where 
 \[
 H^{2n+1}_n=SU(n,1,\A)/SU(n,\A),\quad \A H^n=SU(n,1,\A)/S(U(1,\A)U(n,\A)).
 \]
%{\color{blue}B. Hall's comments: using dualize arguments:

%Recall the fibration structure of $\bS^{4n+3}$ as a $S^3$ bundle over $\mathbb{HP}^n$. Dualize this we obtain that
%the real hyperbolic space $H^{4n+3}$ is a $H^3$ bundle over $\mathbb{H}H^n$. \\ References: Helgason's symmetric spaces book, Borel's book (Semisimple groups and Riemannian symmetric spaces.)

%}  

\subsection{Second method}

The strategy of the second method is similar to the one used in \cite{BD} and appeal to some results proved in \cite{Int-OM}.

%\begin{theorem}
%Let $p_t(r,\eta)$ be the subelliptic heat kernel of $L$ with respect to the Riemannian volume measure $d\mu=\frac{8\pi^{2n+1}}{\Gamma(2n)}(\sinh r)^{4n-1}(\cosh r)^3(\sin\eta)^2drd\eta$.
%Then, one has for $r \ge 0$ and $ \eta \in [0,\pi)$
%\begin{multline*}
%p_t(r,\eta)= \\ \int_{0}^{\infty} \frac{\sinh(u)}{\sin \eta}  \left\{\sum_{m \geq 0} e^{-m(m+2)t }\sin[(m+1)\eta] \sinh[(m+1)u] \right\}q_{t, 4n+3}(\cosh(r)\cosh(u)) du,
%\end{multline*}
%where
%{\color{red}{In the above formula for $p_t(r,\eta)$, one needs to adjust the normalization constant with respect to the measure $d\mu$. Fabrice}}
%\end{theorem}

\begin{proof}
%Recall 
%\begin{equation*}
%L = \frac{\partial^2}{\partial r^2}+((4n-1)\coth r+3\tanh r)\frac{\partial}{\partial r^2}+\tanh^2 r\left(\frac{\partial^2}{\partial \eta^2}+2\cot\eta\frac{\partial}{\partial \eta}  \right).
%\end{equation*}
%We can derive an integral representation of the corresponding heat kernel $p_t(r, \eta)$ as follows. 
Firstly, we decompose the subelliptic heat kernel in the basis of Chebyshev polynomials of the second kind 
\begin{equation*}
U_m(\cos(\eta)) = \frac{\sin((m+1)\eta)}{\sin(\eta)}, \quad m \geq 0, 
\end{equation*}
which are eigenfunctions of $\Delta_{\mathbf{SU}(2)}$: 
\begin{equation*}
\left(\frac{\partial^2}{\partial \eta^2}+2\cot\eta\frac{\partial}{\partial \eta}\right) U_m(\cos(\cdot))(\eta) = -m(m+2) U_m(\cos(\eta)), \quad m \geq 0.
\end{equation*}
Accordingly, 
\begin{equation*}
p_t(r,\eta) = \sum_{m \geq 0} f_m(t,r) U_m(\cos(\eta))
\end{equation*}
where for each $m$, $f_m(t,\cdot)$ solves the following heat equation:
\begin{align*}
\partial_t (f_m)(t,r) & =  \left\{\frac{\partial^2}{\partial r^2}+((4n-1)\coth r+3\tanh r)\frac{\partial}{\partial r^2} -m(m+2) \tanh^2 r\right\}(f_m)(t,r) 
\\& = \left\{\frac{\partial^2}{\partial r^2}+((4n-1)\coth r+3\tanh r)\frac{\partial}{\partial r^2} +\frac{m(m+2)}{\cosh^2 r} -m(m+2) \right\}(f_m)(t,r). 
\end{align*}
Up to a constant $(-m(m+2) - (2n+1)^2)$, the operator 
\begin{equation*}
\mathcal{L}_m := \frac{\partial^2}{\partial r^2}+((4n-1)\coth r+3\tanh r)\frac{\partial}{\partial r^2} +\frac{m(m+2)}{\cosh^2 r} + (2n+1)^2
\end{equation*}
 is an instance of the one considered in \cite{Int-OM} and denoted there $L_{2n}^{\alpha \beta}$ with 
\begin{equation*}
\alpha = 1+\frac{m}{2}, \qquad \beta = - \frac{m}{2}. 
\end{equation*}
From Theorem 2 in \cite{Int-OM}, we deduce that the solution to the wave Cauchy problem associated with the subelliptic Laplacian is given by 
\begin{equation*}
\cos(s\sqrt{-\mathcal{L}_m})(f)(w) = -\frac{\sinh(s)}{(2\pi)^{2n}} \left(\frac{1}{\sinh s}\frac{d}{ds}\right)^{2n} \int_{\mathbb{H}H^n} K_m(s,w,y)f(y)\frac{dy}{(1-||y||^2)^{2n+2}}  
\end{equation*}
where $f \in C_0^{\infty}(\mathbb{H}H^n)$,
\begin{multline*}
K_m(s,w,y) = \frac{(1-\overline{\langle w,y \rangle})^{1+m/2}}{(1-\langle w, y\rangle)^{-m/2}} \frac{\sqrt{\cosh^2(s) - \cosh^2(d(w,y))_+}}{\cosh^2(d(w,y))} 
\\ {}_2F_1\left(m+2, -m, \frac{3}{2}; \frac{\cosh(d(w,y)) - \cosh(s)}{2\cosh(d(w,y))}\right),
\end{multline*}
 ${}_2F_1$ is the Gauss hypergeometric function, and $dy$ stands for the Lebesgue measure in $\mathbb{C}^n$. Using the spectral formula  
\begin{equation*}
e^{tL} = \frac{1}{\sqrt{4\pi t}}\int_{\mathbb{R}} e^{-s^2/(4t)} \cos(s\sqrt{-L}) ds,
\end{equation*}
holding true for any non positive self-adjoint operator, we deduce that the solution to the heat Cauchy problem associated with $\mathcal{L}_m$ is given by: 
\begin{multline*}
e^{t\mathcal{L}_m}(f)(w)  = - \frac{e^{-m(m+2)t-(2n+1)^2t}}{\sqrt{4\pi t}(2\pi)^{2n}}  \int_{\mathbb{R}}ds\sinh(s) e^{-s^2/(4t)} \left(\frac{1}{\sinh s}\frac{d}{ds}\right)^{2n+1} \\ \int_{\mathbb{H}H^n} K_m(s,w,y)f(y)\frac{dy}{(1-||y||^2)^{2n+2}}.
\end{multline*}
Performing $2n+1$ integration by parts in the outer integral we further get: 
\begin{multline*} 
e^{t\mathcal{L}_m}(f)(w) = -\frac{e^{-m(m+2)t-(2n+1)^2t}}{\sqrt{4\pi t}(2\pi)^{2n}}  \int_{\mathbb{R}}ds\sinh(s) \left(\frac{1}{\sinh s}\frac{d}{ds}\right)^{2n+1} e^{-s^2/(4t)} \\ \int_{\mathbb{H}H^n} K_m(s,w,y)f(y)\frac{dy}{(1-||y||^2)^{2n+2}} 
= -\frac{e^{-m(m+2)t-(2n+1)^2t}}{\sqrt{4\pi t}(2\pi)^{2n}}  \int_{\mathbb{H}H^n}f(y)\frac{dy}{(1-||y||^2)^{2n+2}}\\ \int_{\mathbb{R}}ds \sinh(s) K_m(s,w,y) \left(\frac{1}{\sinh s}\frac{d}{ds}\right)^{2n+1} e^{-s^2/(4t)}
= 2\frac{e^{-m(m+2)t-(2n+1)^2t}}{\sqrt{4\pi t}(2\pi)^{2n}}  \int_{\mathbb{H}H^n}f(y)\frac{dy}{(1-||y||^2)^{2n+2}}\\ \int_{d(w,y)}^{\infty}d(\cosh(s))K_m(s,w,y) \left(-\frac{1}{\sinh s}\frac{d}{ds}\right)^{2n+1} e^{-s^2/(4t)}
\end{multline*}
 Recalling the heat kernel on the hyperbolic space $H^{4n+3}$:
\begin{equation*}
q_{t, 4n+3}(\cosh s) := \frac{e^{-(2n+1)^2t}}{(2\pi)^{2n+1}\sqrt{4\pi t}}\left(-\frac{1}{\sinh(s)}\frac{d}{ds}\right)^{2n+1} e^{-s^2/(4t)},
\end{equation*} 
we get
\begin{multline*}
e^{t\mathcal{L}_m}(f)(0) = 4\pi e^{-m(m+2)t}  \int_{\mathbb{H}H^n}f(y)\frac{dy}{(1-||y||^2)^{2n+2}}\\ \int_{d(0,y)}^{\infty}d(\cosh(s))K_m(s,0,y) q_{t, 4n+3}(\cosh(s)).
\end{multline*}
As a result, the subelliptic heat kernel of $\mathcal{L}_m$ reads
\begin{multline*}
\frac{dy}{(1-||y||^2)^{2n+2}}\\ \int_{d(0,y)}^{\infty}d(\cosh(s))K_m(s,0,y) q_{t, 4n+3}(\cosh(s)) = dr \cosh^3(r)\sinh^{4n-1}(r) \\ \int_{r}^{\infty}d(\cosh(s))K_m(s,0,y) q_{t, 4n+3}(\cosh(s)).
\end{multline*}
Performing the variable change $\cosh(s) = \cosh(r)\cosh(u)$ for $u \geq 0$, we transform the last expression into 
\begin{multline*}
dr \cosh^3(r)\sinh^{4n-1}(r) \\ \int_{0}^{\infty} \sinh^2(u) {}_2F_1\left(m+2, -m, \frac{3}{2}; \frac{1- \cosh(u)}{2}\right) q_{t, 4n+3}(\cosh(r)\cosh(u)) du.
\end{multline*}
Since 
\begin{equation*}
(m+1) {}_2F_1\left(m+2, -m, \frac{3}{2}; \frac{1- \cosh(u)}{2}\right) = \frac{\sinh[(m+1)u]}{\sinh(u)},
\end{equation*}
we finally recover the first integral representation displayed in Theorem \ref{IntRep}. 
\end{proof}

%{\color{red}{ Can we simplify the sum $\sum_{m \geq 0} \frac{e^{-m(m+2)t}}{m+1} \sin[(m+1)\eta] \sinh[(m+1)u] $ ? Fabrice}}

\subsection{Relation to the horizontal heat kernel of the complex anti de-Sitter fibration}

In $\mathbb{C}^{2n+1}, n \geq 1$, we consider the signed quadratic form: 
\begin{equation*}
||(z_1, \dots, z_{2n+1})||^2_H = \sum_{i=1}^{2n} |z_i|^2 - |z_{2n+1}|^2. 
\end{equation*}
Then, the  complex anti de-Sitter space $\mathbf{AdS}^{4n+1}(\mathbb{C})$  is the quadric defined by 
\begin{equation*}
\{z = (z_1, \dots, z_{n+1}) \in \mathbb{C}^{n+1}, ||z||_H = -1\},
\end{equation*}
and the circle group $U(1)$ acts on it by $z \mapsto ze^{i\theta}$. This action gives rise to the AdS fibration over the complex hyperbolic space $\mathbb{C}H_n$ with $U(1)$-fibers. 

\begin{proposition}
Let $p_t^{\C}(r,\eta)$ be the subelliptic heat kernel on the $4n+1$-dimensional  complex anti-de Sitter space as defined in \cite{W} and \cite{BD}. Then
\[
-\frac{e^{-4nt}}{2\pi\cosh r\sin\eta}\frac{\partial}{\partial\eta}p_t^{\C}(r, \eta)=p_t(r,\eta).
\]
\end{proposition}
\begin{proof}
From \cite{W} we know that for $\eta\in[-\pi,\pi]$
\[
p_t^{\C}(r, \eta)=\frac{1}{\sqrt{4\pi t}}\sum_{k\in \Z}\int_{-\infty}^{+\infty}e^{\frac{y^2}{4t} }q_t^{4n+1}(\cosh r\cosh (y+i \eta+2k\pi i))dy
\]
where $q_t^{4n+1}$ is the Riemannian heat kernel of the $4n+1$-dimensional hyperbolic space. It is well known that (for instance, see \cite{DM})
\[
q_t(x)=-\frac{e^{-(4n+1)t}}{2\pi}\frac{d}{dx}q_t^{4n+1}(x)
\]
hence we can easily obtain that
\begin{align*}
\frac{\partial}{\partial\eta}p_t^{\C}(r, \eta)&=\frac{1}{\sqrt{4\pi t}} \sum_{k\in \Z}\int_{-\infty}^{+\infty}e^{\frac{y^2}{4t} }\frac{\partial}{\partial\eta} q_t^{4n+1}(\cosh r\cosh (y+i \eta+2k\pi i))dy\\
&=-\frac{2\pi e^{(4n+1)t}\cosh r}{\sqrt{4\pi t}} \sum_{k\in \Z}\int_{-\infty}^{+\infty}e^{\frac{y^2-(\eta+2k\pi)^2}{4t} }\sinh \frac{(\eta+2k\pi) y}{2t}\sinh y\, q_t^{4n+3}(\cosh r\cosh y)dy\\
&=-{2\pi e^{(4n+1)t}\cosh r}\,p_t(r,\eta)
\end{align*}
\end{proof}

\subsection{Small-time asymptotics of the subelliptic heat kernel}
In this section we analyze the short time asymptotic behaviors of the subelliptic kernel. First from \eqref{eq-IntRep} we obtain that the dominating term is $k=0$ as $t\to0$. This can be seen from the second expression of \eqref{eq-s-t} that for $u, \eta \in[0,\pi)$, 
\[
 s_t (\eta,u) =\frac{e^{t}\sinh\frac{\eta u}{2t}}{  \sqrt{\pi t}  \sin \eta \sin u } e^{\frac{-(u^2+\eta^2)}{4t}}  \left(1+O(e^{-C/t}) \right) 
\]
for some constant $C>0$. Hence when $t\to0$, we have that
\begin{equation}\label{eq-pt-leading}
 p_t(r,\eta) =\frac{ e^{t}}{  \sqrt{\pi t}  }  \int_0^{+\infty} \frac{ \sinh y \sin \left(  \frac{ \eta y}{2t}\right) }{\sin \eta} e^{\frac{y^2- \eta^2}{4t}} q_{t,4n+3}( \cosh r\cosh y ) dy+O(e^{-C/t}).
\end{equation}
Before the estimates, let us recall the Riemannian heat kernel $q_t$ has the small time asymptotic
\begin{align}\label{eq9}
q_{t,4n+3}(\cosh\delta')=\frac{1}{(4\pi t)^{2n+\frac{3}{2}}}\left(\frac{\delta'}{\sinh\delta'}\right)^{2n+1}e^{-\frac{\delta'^2}{4t}}\left(1+\left((2n+1)^2-\frac{(2n+1)2n(\sinh\delta'-\delta'\cosh\delta')}{\delta'^2\sinh\delta'}\right)t+O(t^2)\right).
\end{align}
where $\delta' \in [0,\infty)$  is the Riemannian distance and $\cosh \delta'=\cosh r\cosh y$.
\begin{prop}
When $t\to 0$, we have
\[
p_t(0,0)=\frac{1}{(4\pi t)^{2n+3}}(A_n+B_nt+O(t^2)),
\]
where $A_n=4\pi\int_0^{+\infty}\frac{y^{2n+2}}{(\sinh y)^{2n}}dy$ and $B_n=4\pi\int_0^{+\infty}\frac{y^{(2n+2)}}{(\sinh y)^{2n}}\left(4n^2+4n+2-\frac{2n(2n+1)(\sinh y-y\cosh y)}{y^2\sinh y}\right)dy$.
\end{prop}
\begin{proof}
From \eqref{eq-pt-leading} we know  that
\begin{eqnarray*}
p_t(0,0)= \frac{e^{t}}{\sqrt{\pi t}}\int_{0}^{+\infty}\sinh y\cdot e^{\frac{y^2}{4t}}\cdot\frac{y}{2t} q_{t,4n+3}(\cosh y)dy.
\end{eqnarray*}
Plug in \eqref{eq9}, we have that
\[
p_t(0,0)= \frac{4\pi e^{t}}{(4\pi t)^{2n+3}}\int_{-\infty}^{+\infty}\frac{y^{2n+2}}{(\sinh y)^{2n}}\left(1+\left((2n+1)^2-\frac{2n(2n+1)(\sinh y-y\cosh y)}{y^2\sinh y}\right)t\right)dy.
\]
Hence the claimed estimates.
\end{proof}

The small time behavior of the subelliptic heat kernel on the vertical cut-locus, namely the points $(0,\eta)$ that can be achieved by flowing along vertical vector fields is quite different. A short-cut  to its estimate is by differentiating the small time estimate of  $p_t^{\C}(0,\eta)$.
\begin{prop}\label{asy-cut}
For $ \eta\in(0,\pi)$, $t\rightarrow 0$,
\[
p_t(0, \eta)=\frac{1}{4\pi\sin \eta\,{2^{6n}t^{4n+1}(2n-1)!}}\left( (\pi + \eta) \eta^{2n-1}e^{-\frac{2\pi \eta+ \eta^2}{4t}}   \right)(1+O(t)).
\]
\end{prop}
\begin{proof}
Since   
\[
-\frac{e^{-4nt}}{2\pi\cosh r\sin \eta}\frac{\partial}{\partial \eta} p_t^{\C}(r, \eta)=p_t(r,\eta),
\]
we just need to plug in the small time asymptotic of $p_t^{\C}$ on the cut locus. Recall
\[
p_t^{\C}(0, \eta)=\frac{ \eta^{2n-1}}{2^{6n}t^{4n}(2n-1)!}e^{-\frac{2\pi \eta+ \eta^2}{4t}}(1+O(t)),
\]
we then have the conclusion.
\end{proof}

We now deduce the small time behavior of the kernel on the horizontal base space of $\mathbf{AdS}^{4n+3}(\mathbb{H})$. i.e. $(r,0)$, $r\not=0$.
\begin{prop}\label{propr0}
For $r\in(0,\infty)$,  we have
\[
p_t(r,0)=\frac{1}{(4\pi t)^{2n+\frac{3}{2}}}\left(\frac{r}{\sinh r}\right)^{2n+1}e^{-\frac{r^2}{4t}}\left(\frac{1}{r\coth r-1}\right)^{\frac{3}{2}}(1+O(t)).
\]
\end{prop}
%%%%%%%%%%%%%%%%%%%%%%%%%%%%
\begin{proof}
By \eqref{eq-pt-leading} we have
\[
p_t(r,0)=\frac{e^{t}}{\sqrt{4\pi t}}\int_{-\infty}^{+\infty}(\sinh y)\frac{y}{2t}e^{\frac{y^2}{4t}}q_t(\cosh r\cosh y)dy,
\]
and by plugging in (\ref{eq9}), we obtain that 
\[
p_t(r,0)=\frac{1}{(4\pi t)^{2n+2}}(1+O(t))\int_{-\infty}^\infty e^{\frac{y^2-(\cosh^{-1} (\cosh r\cosh y))^2}{4t}}\left(\frac{ y}{2t}\right)\sinh y
\left(\frac{\cosh^{-1} (\cosh r\cosh y)}{\sqrt{\cosh^2r\cosh^2y-1}} \right)^{2n+1}dy
\]
Similarly as in \cite{W}, we can analyze it by the Laplace method and obtain the desired result.
\end{proof}

For the case  $(r,\eta)$,  with $r\not=0$,  we use the steepest descent method. Similarly we obtain
 
\begin{prop}\label{asym}
Let  $r\in(0,\infty)$, $\eta\in[0,\pi)$. % {\color{red} (Notice the range changed compare to $\mathbf{SU}(2)$).} 
Then when $t\to 0$,
\begin{equation}\label{sta}
p_t(r,\eta)=\frac{1}{(4\pi t)^{2n+\frac{3}{2}}} \frac{\sin\varphi(r,\eta)}{\sin\eta\sinh r}\frac{(\cosh^{-1} u(r,\eta))^{2n+1}}{\sqrt{\frac{u(r,\eta)\cosh^{-1} u(r,\eta)}{\sqrt{u^2(r,\eta)-1}}-1}}\frac{e^{-\frac{(\varphi(r,\eta)-\eta)^2\tanh^2 r}{4t\sin^2(\varphi(r,\eta))}}}{(u(r,\eta)^2-1)^{n}}(1+O(t)),
\end{equation}
where $u(r,\eta)=\cos r\cos\varphi(r,\eta)$ and $\varphi(r,\eta)$ is the unique solution in $\left(-\arccos\left( \frac{1}{\cosh r}\right),\arccos\left( \frac{1}{\cosh r}\right)\right)$ to the equation
\begin{equation}\label{varphi}
\varphi(r,\eta)-\eta=\cosh r\sin\varphi(r,\eta)\frac{\cosh^{-1}(\cos \varphi(r,\eta)\cosh r)}{\sqrt{\cosh^2r\cos^2\varphi(r,\eta)-1}},
\end{equation}

\end{prop}
\begin{proof}
From Lemma 3.6 in \cite{W} we know that 
\[
f(y)=(\cosh^{-1} (\cosh r \cosh y))^2-(y-i\eta)^2
\]
has a critical point at $i\varphi(r,\eta)$ and 
\[
f^{''}(i\varphi(r,\eta))=\frac{2\sinh^2 r}{u(r,\theta)^2-1}\left(\frac{u(r,\theta)\cosh^{-1} u(r,\theta)}{\sqrt{u^2(r,\theta)-1}}-1 \right),
\]
is positive, where $u(r,\theta)=\cosh r\cos\varphi(r,\theta)$. Also
\begin{equation}\label{eq-f-i-varphi}
f(i\varphi(r,\theta))=(\cosh^{-1}(\cosh y\cos \varphi))^2-(\varphi-\theta)^2=\frac{(\varphi(r,\theta)-\theta)^2\tanh^2 r}{\sin^2(\varphi(r,\theta))}.
\end{equation}
Then we follow the idea in \cite{W} and use steepest descent method to obtain  the desired conclusion.
\end{proof}

\section{The subelliptic heat kernel on the twistor space of $\mathbb{H}H^n$}\label{sec-twistor}

\subsection{Radial part of the sub-Laplacian on $\mathbb{C}H_1^{2n+1}$}

Besides the action of  $\mathbf{SU}(2)$ on $\mathbf{AdS}^{4n+3}(\mathbb{H})$ that induces the quaternionic anti de-Sitter fibration which was studied in the previous sections, we can also consider the action of $\mathbb{S}^1$ on $\mathbf{AdS}^{4n+3}(\mathbb{H})$ that induces the  fibration:
\begin{align*}
\mathbb{S}^1 \to  \mathbf{AdS}^{4n+3}(\mathbb{H})  \to  \mathbb{C}H_1^{2n+1},
\end{align*}
where we simply define $ \mathbb{C}H_1^{2n+1}$ as the complex pseudo-hyperbolic space $ \mathbf{AdS}^{4n+3}(\mathbb{H})  / \mathbb{S}^1$. We refer to  \cite{BAI} for a general definition of complex pseudo-hyperbolic spaces. Note that the metric on $ \mathbb{C}H_1^{2n+1}$ has signature $(4n,2)$. We can then see $\mathbb{S}^1$ as a subgroup of $\mathbf{SU}(2)$ and have the classical Hopf  fibration
\begin{align*}
 \mathbb{S}^1  \to \mathbf{SU}(2) \to \mathbb{CP}^1
\end{align*}

We can therefore construct the following commutative fibration diagram
\begin{diagram}\label{diag}
  & & \mathbb{S}^1 & & \\
  & \ldTo & \dTo & & \\
 \mathbf{SU}(2) & \rTo &\mathbf{AdS}^{4n+3}(\mathbb{H}) & \rTo & \mathbb{H}H^n \\
 \dTo & &\dTo & \ruTo & \\
 \mathbb{CP}^1 & \rTo & \mathbb{C}H_1^{2n+1} & & 
\end{diagram}

The pseudo-Riemannian submersion $\mathbb{C}H_1^{2n+1} \to \mathbb{H}H^n $ obtained  at the bottom of the diagram is similar to Example 4, page 4 in \cite{BAI} and the fibration
\[
\mathbb{CP}^1 \to \mathbb{C}H_1^{2n+1} \to \mathbb{H}H^n
\]
shows that $\mathbb{C}H_1^{2n+1}$ is therefore the twistor space of the quaternionic contact manifold $\mathbb{H}H^n$.

We consider the sub-Laplacian ${\mathcal{L}}$ on  $\mathbb{C}H_1^{2n+1} $, it is then the lift of the Laplace-Beltrami operator of  $\mathbb{H}H^n$. From the above diagram, ${\mathcal{L}}$ is also the projection of the sub-Laplacian $L$ of $\mathbf{AdS}^{4n+3}(\mathbb{H}) $ on $\mathbb{C}H_1^{2n+1} $. As we have seen before,  the radial part of $L$ is 
\[
L=\frac{\partial^2}{\partial r^2}+((4n-1)\coth r+3\tanh r)\frac{\partial}{\partial r}+\tanh^2r (\frac{\partial^2}{\partial \eta^2}+2\cot \eta\frac{\partial}{\partial \eta})
\]
where $r$ is the radial coordinate on $\mathbb{H}H^n$ and $\eta$ the radial coordinate on $ \mathbf{SU}(2) $. The operator
\[
\Delta_{\mathbf{SU}(2)}=\frac{\partial^2}{\partial \eta^2}+2\cot \eta\frac{\partial}{\partial \eta}
\]
is the radial part of the Laplace-Beltrami operator on $ \mathbf{SU}(2) $. As it has been proved in Baudoin-Bonnefont (see \cite{BB}), by using the Hopf fibration,
\begin{align*}
\mathbb{S}^1 \to  \mathbf{SU}(2)  \to  \mathbb{CP}^{1}
\end{align*}
we can write
\[
\frac{\partial^2}{\partial \eta^2}+2\cot \eta\frac{\partial}{\partial \eta}=\frac{\partial^2}{\partial \phi^2}+2\cot 2\phi\frac{\partial}{\partial \phi}+(1+\tan^2\phi )\frac{\partial^2}{\partial \theta^2}
\]
where $\phi$ is the radial coordinate on $\mathbb{CP}^{1}$ and $\frac{\partial}{\partial \theta}$ the generator of the action of $\mathbb{S}^1$ on $\mathbf{SU}(2)$.
Therefore the radial part of the sub-Laplacian $\mathcal{L}$ on  $\mathbb{C}H_1^{2n+1}$ is given by 
\begin{equation}\label{L-CP}
\mathcal{L}=\frac{\partial^2}{\partial r^2}+((4n-1)\coth r+3\tanh r)\frac{\partial}{\partial r}+\tanh^2r\left( \frac{\partial^2}{\partial \phi^2}+2\cot 2\phi\frac{\partial}{\partial \phi}\right),
\end{equation}
and the invariant measure, up to a normalization constant,  is $(\sinh r)^{4n-1} (\cosh r)^3 \sin 2\phi dr d\phi$.

\subsection{Integral representation of the subelliptic heat kernel}

From \eqref{L-CP} we notice that $\frac{\partial^2}{\partial \phi^2}+2\cot 2\phi\frac{\partial}{\partial \phi}$ is the radial part of the Laplacian on $\mathbb{CP}^1$.  It is known that the eigenfunction associated to the eigenvalue $-4m(m+1)$ is given by $P_m^{0,0}(\cos2\phi)$ where $P_m^{0,0}$ is the Legendre polynomial
\[
P_m^{0,0}(x)=\frac{(-1)^m}{2^m m!}\frac{d^m}{dx^m}(1-x^2)^{m}.
\]
Moreover, the heat kernel of $\frac{\partial^2}{\partial \phi^2}+2\cot 2\phi\frac{\partial}{\partial \phi}$ is given by
\[
u_t (\phi_1,\phi_2)=\sum_{m=0}^{+\infty} (2m+1) e^{-4m(m+1)t} P_m^{0,0}(\cos2\phi_1)P_m^{0,0}(\cos2\phi_2)
\]

By using the same methods as before, we obtain:

\begin{theorem}
Let $h_t(r,\phi)$ be the subelliptic heat kernel of $\mathcal L$.
Then, one has for $r \ge 0$ and $ \phi \in [0,\pi)$
 \begin{align*}
 & h_t(r,\phi)\\
 = & \int_{0}^{\infty} (\sinh u)^2 \left\{\sum_{m=0}^{+\infty} (2m+1) e^{-4m(m+1)t} P_m^{0,0}(\cos2\phi )P_m^{0,0}(\cosh 2u)\right\}q_{t, 4n+3}(\cosh r \cosh u) du.
\end{align*}
where
\begin{equation*}
q_{t, 4n+3}(\cosh s) := \frac{e^{-(2n+1)^2t}}{(2\pi)^{2n+1}\sqrt{4\pi t}}\left(-\frac{1}{\sinh(s)}\frac{d}{ds}\right)^{2n+1} e^{-s^2/(4t)},
\end{equation*} 
is the heat kernel on the real $4n+3$ dimensional real hyperbolic space.

\end{theorem}

\section{Further developments}\label{sec-further}
\subsection{Quaternionic magnetic Laplacian} 
The second method we used to compute the sub elliptic heat kernel appeal to an operator $L_n^{\alpha\beta}, \alpha, \beta \in \mathbb{R},$ introduced in \cite{Int-OM}. When $\alpha = -\beta$, this is the radial part of the so-called generalized Maass Laplacian and it reduces when $n=1$ to the complex-hyperbolic magnetic Laplacian. Moreover, it was noticed in \cite{BD} that the generalized Maass Laplacian coincides with the partial Fourier transform of the horizontal Laplacian of the complex AdS space with respect to the $U(1)$-fiber coordinate (the dual variable of the Fourier transform is then interpreted as the magnetic field strength). 

On the other hand, a quaternionic magnetic Laplacian with uniform field on $\mathbb{H}$ was defined and studied in \cite{IKZ} and \cite{IKZ1}. There, the author consider the one-form $A$ (magnetic potential): 
\begin{multline*}
2A = -(B_1x+B_2y+B_3z) dt + (B_1t-B_3y+B_2z)dx  \\ + (B_2t+B_3x-B_1z)dy + (B_3t-B_2x+B_1y)dz
\end{multline*}
whose exterior derivative (curvature) is the self-dual (with respect to Hodge operator) two-form corresponding to the uniform magnetic field $B = (B_1,B_2, B_3) \in \mathbb{R}^3$: 
\begin{multline*}
dA = B_1(dt \wedge dx + dy \wedge dz) + B_2(dt \wedge dy + dz \wedge dx) + B_3(dt \wedge dz + dx \wedge dy). 
\end{multline*}
The quaternionic magnetic Laplacian then defined by 
\begin{equation*}
-[(\partial_t + iA_0)^2 + (\partial_x + iA_1)^2 + (\partial_y +iA_2)^2 +(\partial_z+iA_3)^2]
\end{equation*}
where $A_i, 0 \leq i \leq 3$ are the components of $A$. This operator was then identified  as the partial Fourier transform of the horizontal Laplacian of the quaternionic Heisenberg group with respect to the vertical coordinates. Note that if we consider the quaternionic symplectic form 
\begin{equation*}
\omega:= \frac{1}{2}(dw\overline{w} - wd\overline{w}) := \omega_1I+\omega_2J+\omega_3K
\end{equation*}
on $\mathbb{H}$ and if $B = B_1I+B_2J+B_3K$, then the above quaternionic magnetic Laplacian may be written as a Bochner-type Laplacian: 
\begin{equation*}
-(d-i \Re(B\omega))^{\star}(d-i \Re(B\omega))
\end{equation*}
acting on functions, where $d$ is the exterior derivative and $\star$ is the adjoint operator with respect to the flat Riemannian metric on $\mathbb{H} \approx \mathbb{R}^4$. Accordingly, we may define the quaternionic analogue of the generalized Maass Laplacian as the partial Fourier transform of the horizontal Laplacian displayed in \eqref{SubLapl} with respect to the fiber variables $(\phi_1, \phi_2, \phi_3)$. However, note that in contrast with the flat setting, we need to add a weight when we perform the partial Fourier transform in order to neutralize the factor 
\begin{equation*}
\frac{1}{\cos^2\eta} = 1+\tan^2 \eta = 1+\phi_1^2 + \phi_2^2 + \phi_3^2,
\end{equation*}
which amounts to get the translation invariance of the sublaplacian with respect to $(\phi_1, \phi_2, \phi_3)$. Moreover, since 
\begin{equation*}
V_i = \frac{\partial}{\partial w_i} - \frac{1}{\cos^2\eta}\zeta\left(\frac{\partial}{\partial w_i}\right)\frac{\partial}{\partial \phi}, 
\end{equation*} 
where we recall that $\zeta$ is defined in \eqref{QKah}, it would be interesting to check whether the magnetic Laplacian may be written or not as a Bochner-type Laplacian acting on functions.

\subsection{The heat kernel of a sub-d'Alembertian on $H^{4n+3}$}
We may compute the heat kernel  of the sub-d'Alembertian on $H^{4n+3}$ which is given by
\[
\mathcal{L} = \Delta_{H^{4n+3}}-\Delta_P=\frac{\partial^2}{\partial r^2}+((4n-1)\coth r+3\tanh r)\frac{\partial}{\partial r^2}-\tanh^2 r\left(\frac{\partial^2}{\partial \eta^2}+2\coth\eta\frac{\partial}{\partial \eta} \right)
\]
This operator is obtained from the horizontal Laplacian on the quaternionic anti-de Sitter space by complexification of the fiber $\mathbf{SU}(2)$.
Indeed, the spectrum of the `fiber' part is purely continuous: 
\begin{equation*}
\left(\frac{\partial^2}{\partial \eta^2}+2\coth\eta\frac{\partial}{\partial \eta}\right) \phi_{\lambda}^{(1/2, -1/2)}(\eta) = -(\lambda^2+1)\phi_{\lambda}^{(1/2, -1/2)}(\eta), \quad \lambda \in \mathbb{R},  
\end{equation*}
where (\cite{Koor})
\begin{equation*}
\phi_{\lambda}^{(1/2, -1/2)}(\eta) = {}_2F_1\left(\frac{1 + i\lambda}{2}, \frac{1 - i\lambda}{2}, \frac{3}{2}; -\sinh^2(\eta)\right) 
\end{equation*}
is the Jacobi function of parameters $(1/2,-1/2)$. Expanding the heat kernel in the $\eta$-variable as a inverse Fourier-Jacobi transform of some smooth function $(t,r) \mapsto v_{\lambda}(t,r)$ (\cite{Koor}): 
\begin{equation*}
\int_{\mathbb{R}} v_{\lambda}(t,r) \phi_{\lambda}^{(1/2, -1/2)}(\eta) \eta^2d\eta, 
\end{equation*}
the function $v_{\lambda}$ should solve the heat equation associated with the operator 
\[
\frac{\partial^2}{\partial r^2}+((4n-1)\coth r+3\tanh r)\frac{\partial}{\partial r^2} + (\lambda^2+1) \tanh^2 r. 
\]
Using the identity $\tanh^2(r) = 1-1/\cosh^2(r)$, then we are led to 
\[
\left(\frac{\partial^2}{\partial r^2}+((4n-1)\coth r+3\tanh r)\frac{\partial}{\partial r^2} - \frac{\lambda^2+1}{\cosh^2 r} + (\lambda^2+1)\right)v_{\lambda}(t,r) = \partial_t v_{\lambda}(t,r). 
\]
The corresponding heat kernel may then be derived along the same lines written in the second method by choosing the parameters $\alpha, \beta$ such that:
\begin{equation*}
\alpha + \beta = 1, \quad 4\alpha \beta = \lambda^2 + 1, 
\end{equation*}
that is, $\alpha = 1-\beta = (1+i\lambda)/2$. 

%Note that $\mathcal{L}$ is the operator associated to $L$ through Wick's rotations:
%\[
%L(f\circ \tau)=\mathcal{L} f \circ \tau
%\]
%and may be interpreted as the dual of our sub-Laplacian.

\end{document}